\documentclass[a4paper, 10pt]{amsart}
\usepackage[margin=3cm]{geometry}
\usepackage{amsmath, amsfonts, amssymb}
\usepackage{amsthm}
\usepackage{esint}
\usepackage[
bookmarks=false]{hyperref}
\usepackage{cite}
\usepackage{array}

\usepackage{amsaddr}
\usepackage{enumerate}
\usepackage{pgfplots}
\usepackage{cleveref}

\usepackage [english]{babel}
\usepackage [autostyle, english = american]{csquotes}
\MakeOuterQuote{"}

\title{An extremality property of Szeg\H{o} projections\\ on Heisenberg groups}

\author{Gian Maria Dall'Ara}
\address{Istituto Nazionale di Alta Matematica ``F. Severi"\\ Research Unit Scuola Normale Superiore\\
	Piazza dei Cavalieri, 7, 56126 Pisa, Italy}
\email{dallara@altamatematica.it}

\author{Bernhard Lamel}
\address{Science Program, Texas A\& M University at Qatar \\ Education City, Doha, Qatar }
\email{bernhard.lamel@qatar.tamu.edu, bernhard.lamel@univie.ac.at}

\date{\today}

\newcommand{\R}{\mathbb{R}}
\newcommand{\N}{\mathbb{N}}

\newcommand{\C}{\mathbb{C}}

\newcommand{\eps}{\varepsilon}
\newcommand{\CR}{\mathrm{CR}}
\newcommand{\Sz}{\mathcal{S}}
\newcommand{\vnorm}[1]{\left\| #1 \right\|}

\newtheorem{thm}{Theorem}
\newtheorem*{thm'}{Theorem 1'}
\newtheorem{lem}[thm]{Lemma}
\newtheorem{prop}[thm]{Proposition}
\newtheorem{dfn}[thm]{Definition}

\newtheorem*{ex}{Example}
\newtheorem{rmk}[thm]{Remark}

\DeclareMathOperator{\real}{Re}
\DeclareMathOperator{\imag}{Im}

\begin{document}

\begin{abstract}
We prove that Heisenberg groups, a.k.a.~the boundaries of Siegel domains, minimize the $L^p$ operator norm of the Szegő projection in a large class of weighted CR manifolds of hypersurface type. 
\end{abstract}

\maketitle

\section{Introduction} 
\label{sec:introduction}

The classical Szegő projection associates to every $L^2$ function on the boundary of a domain in $\C^n$ its orthogonal projection onto the space of boundary values of holomorphic functions, that is, the Hardy space $H^2$. For a variety of pseudoconvex domains, the corresponding Szegő projection is known to be a singular integral operator (see, e.g., \cite{phong_stein, nagel_rosay_stein_wainger, mcneal_stein, charpentier_dupain}), whose $L^p$ mapping behaviour is therefore of considerable interest from the point of view of harmonic and complex analysis alike. 

In this note, we compare the $L^p$ mapping behaviour of the Szegő projection $\mathcal{S}$ on a real hypersurface $M\subset \C^{n+1}$ with the $L^p$ mapping behaviour of the Szegő projection $\mathcal{S}_n$ on the model strictly pseudoconvex hypersurface of the same dimension, namely the Heisenberg group
\[ \mathbb{H}^n=\{z\in \C^{n+1}\colon  \imag z_{n+1} = \left| \hat z \right|^2\},\qquad \hat z= (z_1, \dots ,z_n).  \]
The quantities of interest are the operator norms
\[ \left\| \mathcal{S} \right\|_{p\to p} := \sup_f
\frac{\vnorm{\mathcal{S} f}_p}{\vnorm{f}_p}, \quad \left\| \mathcal{S}_n \right\|_{p\to p} := \sup_f 
\frac{\vnorm{\mathcal{S}_n f}_p}{\vnorm{f}_p} \qquad (1\leq p\leq \infty),  \]
where the Szegő projections and the $L^p$ norms $\vnorm{\cdot}_p$ are defined with respect to a fixed background measure, which we always assume to have a smooth positive density with respect to Lebesgue measure. The background measure on the model $\mathbb{H}^n$ is Lebesgue measure in the global coordinates $(\hat z, \real z_{n+1})$. 

Our main result shows that the model Szegő projection $\mathcal{S}_n$ minimizes the $L^p$ operator norm in a large class of hypersurfaces $M$. 

\begin{thm}\label{thm:intro}
If $M$ is compact and pseudoconvex, then 
\[ \left\| \mathcal{S} \right\|_{p\to p} \geq 
 \left\| \mathcal{S}_n \right\|_{p\to p} . \]
The same conclusion holds, more generally, if $M$ is a $(2n+1)$-dimensional CR manifold of hypersurface type satisfying property C at some strongly pseudoconvex point $x_0$. 
\end{thm}

The definition of property C is given in Section \ref{sec:basics}. This property depends both on the CR structure and the choice of background measure, and it is automatically satisfied by every compact pseudoconvex CR submanifold of $\C^N$ (not necessarily a real hypersurface, see Remark \ref{rmk:property_C} below) and by the model itself (see Proposition \ref{heis_propC}). 

The idea of the proof is pretty straightforward: since $M$ is well-approximated by the model $\mathbb{H}^n$ near a strongly pseudoconvex point $x_0$, whenever $f\in L^2(M)$ is sharply localized around $x_0$ its Szegő projection $\Sz f$ can be well-approximated by $\Sz_n \tilde{f}$, where $\tilde{f}$ is a function obtained "transplanting" $f$ on the model. 

Yet, there is a difficulty. We know a priori that \emph{such an approach cannot work on every, even strongly pseudoconvex, CR manifold}. In fact, there are compact strongly pseudoconvex CR structures on the three-dimensional torus $\mathbb{T}^3$ with the property that the only square-integrable CR functions are the constants \cite{barrett}. For such a pathological CR manifold $M$, the associated Szegő projection is the averaging operator $\Sz f=\frac{1}{\nu(M)}\int_Mf\, d\nu$, whose $L^p$ operator norm is $\lVert \Sz\rVert_{p\rightarrow p}=1$ for every $p\in [1,\infty]$. Since $\Sz_n$ is unbounded on $L^1$ (see, e.g., \cite[Chapter 12]{stein}), the conclusion of Theorem \ref{thm:intro} cannot hold for $M$. Thus, some additional hypothesis, like our property C, and the ensuing twists appearing in our argument, are indeed necessary. 

A few more remarks may be of interest. \begin{enumerate}
	\item We do not know whether a minimizer for the $L^p$ operator norm exists in the restricted class of compact pseudoconvex embeddable CR manifolds of hypersurface type (of a fixed dimension). \emph{It is natural to ask whether the standard CR sphere, endowed with a rotation invariant measure, is such a minimizer}. Since $\mathbb{H}^n$ and the punctured sphere $\mathbb{S}^{2n+1}\setminus\{*\}$ are CR isomorphic (see, e.g., \cite[Chapter 12]{stein}) and points have zero capacity in dimension $3$ or higher, the Szegő projection on the standard CR sphere w.r.t.~a rotation invariant measure is equivalent to the Szegő projection $\widetilde\Sz_n$ on $\mathbb{H}^n$ w.r.t.~a certain (finite) background measure, different from Lebesgue measure. 	
	\item To the best of our knowledge, the constants $\lVert \Sz_n\rVert_{p\rightarrow p}$ (except for the "trivial" cases $p=1,\infty$) are unknown. C.~Liu showed \cite{liu} that $\lVert \Sz_n\rVert_{p\rightarrow p}\geq \frac{\Gamma\left(\frac{n+1}{p}\right)\Gamma\left(\frac{n+1}{q}\right)}{\Gamma\left(\frac{n+1}{2}\right)^2}$, where $q$ is the conjugate exponent to $p$, and conjectures that equality holds. 
	\item Since $\lVert \Sz_n\rVert_{1\rightarrow 1}=\infty$, Theorem \ref{thm:intro} implies that the Szegő projection is unbounded on $L^1$ for a variety of CR structures and background measures. More flexible arguments are available to investigate the $L^1$ (un)boundedness of projection operators onto spaces of solutions of first-order PDEs, see \cite{dallara_L1}. 
	\item On a more speculative note, it is reasonable to expect that if $M$ satisfies property C at some \emph{weakly} pseudoconvex point $x_0$, then $\lVert \Sz\rVert_{p\rightarrow p}\geq \lVert \widetilde{\Sz}\rVert_{p\rightarrow p}$ for an appropriate model Szegő projection $\widetilde{\Sz}$ depending on the nature of the point $x_0$. E.g., if $M$ is three-dimensional, pseudoconvex, and $x_0$ has type $2m\geq 2$, then the appropriate model should be $\{\imag z_2=\varphi(z_1, \overline{z_1})\}\subset \C^2$, where $\varphi$ is a subharmonic (nonharmonic) homogeneous polynomial of degree $2m$. If $\Sz_\varphi$ is the projector on this model, a natural follow-up question is whether $\lVert\Sz_\varphi\rVert_{p\rightarrow p}>\lVert\Sz_{|z_1|^2}\rVert_{p\rightarrow p}=\lVert\Sz_1\rVert_{p\rightarrow p}$ whenever $\varphi$ has degree $4$ or higher. If this were the case, equality in Theorem \ref{thm:intro} (in an appropriate class of manifolds) could only be achieved in the strongly pseudoconvex case. This would be a first step in the understanding of possible extremizers, other than the model, of the inequality of Theorem \ref{thm:intro}. 
	\end{enumerate}

The paper is organized as follows: in Section \ref{sec:basics} we rigorously define Szegő projections on "weighted" CR manifolds and discuss property C, while in Section \ref{sec:proof} we prove the main theorem exploiting a couple of preliminary lemmas presented in Section \ref{sec:prelim}. 

\subsection{Acknowledgments}

This research has been funded by the FWF-project P28154 and the European Union’s Horizon 2020 Research and Innovation Programme under the Marie Sklodowska-Curie Grant Agreement No. 841094. The first-named author would like to thank Alessio Martini for an interesting conversation about "transplantation" proofs in analysis. 

\section{Szegő projections on weighted $\CR$ manifolds}
\label{sec:basics}

We assume that the reader is familiar with the basics of CR manifolds, for which we refer to \cite{ber, dragomir_tomassini}. We limit ourselves to recall a few notions, mostly to establish notation.

A pair $(M,T_{1,0}M)$ is said to be a \emph{CR manifold of hypersurface type} if \begin{itemize}
	\item[i.] $M$ is a $(2n+1)$-dimensional connected and orientable real smooth manifold;
    \item[ii.] $T_{1,0}M$ is a rank $n$ vector subbundle of the complexified tangent bundle $\C TM$ such that $T_{1,0}M\cap \overline{T_{1,0}M}=0$ and $[T_{1,0}M,T_{1,0}M]\subseteq T_{1,0}M$, i.e., the commutator of smooth sections of $T_{1,0}M$ is again a section of $T_{1,0}M$.
	\end{itemize}
A $(2n+1)$-dimensional real smooth submanifold $M$ of $\C^N$ with the property that $T_{1,0}M:=\C TM\cap T_{1,0}\C^N$ has rank constantly equal to $n$ is called a \emph{CR submanifold of hypersurface type}: the pair $(M,T_{1,0}M)$ is a CR manifold of hypersurface type. 

In what follows, we will usually omit the specification "of hypersurface type" and the bundle $T_{1,0}M$ from the notation, and we will just say that $M$ is a $\CR$ manifold, or a CR submanifold of some $\C^N$.

The CR manifold $M$ is \emph{pseudoconvex} if there exists a global nowhere vanishing purely imaginary one-form $\theta$ annihilating $T_{1,0}M\oplus \overline{T_{1,0}M}$ and such that $d\theta (L,\overline{L})\geq 0$ for every section $L$ of $T_{1,0}M$. A point $x_0\in M$ is said to be \emph{strongly pseudoconvex} if $d\theta (L,\overline{L})_{x_0}> 0$ for every $L$ that does not vanish in $x_0$. 

\begin{ex}\label{heis_ex}
In this paper $\mathbb{H}^n$, which will be simply called the Heisenberg group, is the $\CR$ manifold obtained endowing $\C_z^n\times \R_t$ with the $\CR$ structure bundle generated by the complex vector fields \begin{equation*}
L_j:=\partial_{z_j}+i\overline z_j\partial_t \qquad j=1,\ldots, n.
\end{equation*}
See \cite[Chapter 12]{stein} for details on the nilpotent Lie group structure of $\mathbb{H}^n$ and its ties with complex analysis. The Heisenberg group is everywhere strongly pseudoconvex. Our model weighted $\CR$ manifold will be $(\mathbb{H}^n,\sigma)$, where $\sigma=\left(\frac{i}{2}\right)^ndz_1\wedge d\overline{z}_1\wedge  \ldots\wedge dz_n\wedge d\overline{z}_n\wedge dt $ is Lebesgue measure on $\C^n \times \R$. 
\end{ex}

If $f\in L^1_{\mathrm{loc}}(M)$, we say that $f$ is a $\CR$-function if $\overline{L}f=0$ in the sense of distributions for every smooth section $L$ of $T_{1,0}M$. For the convenience of the reader, 
we recall what this means in our context. 

If $L$ is a smooth vector field with complex coefficients, that is, a smooth section of the complexified tangent bundle $\C TM$, then we denote by $L^\dagger$ the adjoint of $L$ with respect to the natural pairing of smooth functions with top-degree forms. More precisely, there is a unique first order differential operator $L^\dagger: \Omega_c^{2n+1}M\rightarrow \Omega_c^{2n+1}M$ such that \begin{equation*}
\int_M Lf \cdot \omega = \int_M fL^\dagger\omega\qquad \forall f\in C^\infty(M),\quad\forall\omega\in \Omega_c^{2n+1}M, 
\end{equation*} 
where $\Omega_c^nM$ is the space of smooth compactly supported $(2n+1)$-forms on $M$. 
If $L=\sum_{j=1}^{2n+1}a_j(x)\partial_{x_j}$ and $\omega=g(x)dx_1\wedge \ldots \wedge dx_{2n+1}$ in a local coordinate system, then $L^\dagger\omega=-\sum_{j=1}^{2n+1}a_j(x)\partial_{x_j}g-\left(\sum_{j=1}^{2n+1}\partial_{x_j}a_j(x)\right)g$. 
If $f\in L^1_{\mathrm{loc}}(M)$, then one says that $Lf=0$ in the sense of distributions if\begin{equation*}
\int_M fL^\dagger\omega =0 \qquad \forall \omega\in \Omega_c^{2n+1}M.
\end{equation*}

A \emph{weighted $\CR$ manifold} is a pair $(M,\nu)$, where $M$ is a $\CR$ manifold and $\nu$ is a smooth positive $(2n+1)$-form on $M$. Of course, $\nu$ may be thought of as a Borel measure on $M$ with smooth positive density with respect to Lebesgue measure in any coordinate chart.

Given a weighted $\CR$ manifold $(M,\nu)$, we denote by $\CR^2(M,\nu)$ the space of $\CR$ functions that are square-integrable with respect to $\nu$, with the usual identification of almost everywhere equal functions.

\begin{prop}
	$\CR^2(M,\nu)$ is a closed subspace of the Hilbert space $L^2(M,\nu)$. 
\end{prop}

This follows immediately from the definition of $\CR$ functions and the fact that if a sequence of functions $\{f_k\}_k$ converges in $L^2(M,\nu)$ to a function $f$, then 
$\lim_{k\rightarrow+\infty}\int_Mf_k\omega=\int_Mf\omega$ for every $\omega\in \Omega_c^{2n+1}M$. 

Thus, we may define the main object of study of the present paper.

\begin{dfn}
The Szegő projection $\Sz_{M,\nu}$ of the weighted CR manifold $(M,\nu)$ is 
defined as the orthogonal projection operator mapping $L^2(M,\nu)$ onto $\CR^2(M,\nu)$. We denote by $\Sz_n$ the Szeg\"o projection associated to the Heisenberg group $(\mathbb{H}^n,\sigma)$. 
\end{dfn}We are interested in the $L^p$-operator norms of Szegő projections, namely the quantities defined by \begin{equation*}
N_p(M,\nu):=\sup\left\{\left(\int_M|\Sz f|^pd\nu\right)^\frac{1}{p}\colon f\in L^2(M,\nu)\text{ and }\int_M| f|^pd\nu=1\right\},
\end{equation*}
for $p\in [1,+\infty)$, and by\begin{equation*}
N_\infty(M,\nu):=\sup\left\{\lVert\Sz f\rVert_\infty \colon f\in L^2(M,\nu)\text{ and }\lVert f\rVert_\infty=1\right\}.
\end{equation*}

To formulate our main result, we need the notion of Property C (C is for ``compactness''). Recall that $\overline\partial_b$ is the operator mapping $f\in C^\infty(M)$ into the smooth section of $B_{0,1}M$, the dual bundle of $T_{0,1}M$, defined by $\langle\overline\partial_b f, \overline{L}\rangle=\overline Lf$ for every section $L$ of $T_{1,0}M$. Given a smooth positive measure $\nu$ on $M$ and a Hermitian metric $h$ on $B_{0,1}M$, we have the associated quadratic form \begin{equation}\label{eq:CR_energy}
\mathcal{E}(f)=\int_M |f|^2\, d\nu+\int_M|\overline\partial_bf|_h^2\, d\nu.\end{equation}
Given a precompact open set $B\subseteq M$, we denote by $\mathfrak{D}(B)\subseteq L^2(B)$ the completion of $C^\infty_c(B)$ with respect to $\mathcal{E}$. Notice that different choices of $\nu$ and $h$ produce isomorphic topological vector spaces, because the corresponding quadratic forms $\mathcal{E}$ are comparable on test functions supported on $B$. 

\begin{dfn}[Property C]\label{propC}
	We say that the weighted $\CR$ manifold $(M, \nu)$ satisfies property C at $x_0\in M$ if there exists a precompact open neighborhood $B$ of $x_0$ such that the operator \begin{eqnarray*}
		\mathfrak{D}(B)&\longrightarrow& L^2(B)\\
f&\longmapsto& 1_B\left(1-\mathcal{S}_{M,\nu}\right)f
	\end{eqnarray*}
is compact (by the observation right before this definition, compactness is independent of the choice of metric on $B_{0,1}M$). 
\end{dfn}

\begin{rmk}\label{rmk:property_C}
Property C	is a local compactness condition of the kind playing an important role in the theories of the $\overline\partial$ and the $\overline\partial_b$ problem (see, e.g., \cite[Chapter 4]{straube}). By Sobolev embedding, it holds whenever a subelliptic estimate of the form \[
\lVert u\rVert_{W^\epsilon(B)}\leq C \lVert\overline\partial_b u\rVert_{L^2(M)}
\] holds for every $u\in L^2(M)$ orthogonal to $\CR^2(M,\nu)$, where $\epsilon>0$. Such a subelliptic estimate with $\epsilon=\frac{1}{2}$ is known to hold in a small enough neighborhood $B$ of a strongly pseudoconvex point $x_0\in M$, under the additional assumptions that $M$ is compact, pseudoconvex, and that the maximal $L^2$ extension of $\overline\partial_b$ has closed range. We refer to \cite{kohn_estimates} for this and far more general results on these matters. 

We point out that, if $M$ is a compact pseudoconvex CR submanifold of $\C^N$, then the closed range property automatically holds \cite{baracco, baracco2}. Since such a CR submanifold necessarily has a point of strong pseudoconvexity (viz. any point at maximal distance from the origin of $\C^N$), we see that every compact pseudoconvex CR submanifold of $\C^N$ has a strongly pseudoconvex point at which property C holds, that is, satisfies the assumption of Theorem 1' below. 
	\end{rmk}

We can now restate the main theorem in the more precise notation of this section. 

\begin{thm'}\label{thm:main}
	Let $(M,\nu)$ be a weighted $\CR$ manifold of dimension $2n+1$. If $(M,\nu)$ satisfies property C at a strongly pseudoconvex point $x_0\in M$, then we have \begin{equation*}
	N_p(M,\nu)\geq N_p(\mathbb{H}^n,\sigma)\qquad\forall p\in [1,+\infty].
	\end{equation*} 
\end{thm'}

\section{Preliminaries}\label{sec:prelim}

\subsection{Folland--Stein coordinates}

Define the one-parameter group of \emph{parabolic scalings} on $\C^n\times \R$ as follows: \begin{equation*}
\Phi_\lambda(z,t):=(\lambda z, \lambda^2 t)\qquad (\lambda >0).
\end{equation*}
We say that a monomial $z^\alpha \overline{z}^\beta t^\gamma$, where $\alpha, \beta\in \N^n$ and $\gamma\in \N$, has \emph{parabolic weight} \[\sum_{j=1}^n(\alpha_j+\beta_j)+2\gamma,\] and that a smooth function defined in a neighborhood of $(0,0)\in \C^n\times \R$ has parabolic weight $\geq w$ if every monomial with nonzero coefficient in its Taylor expansion has weight $\geq w$. The parabolic weight of a vector field $E=\sum_{j=1}^n\left\{a_j(z,t)\partial_{z_j}+b_j(z,t)\partial_{\overline{z}_j}\right\}+c(z,t)\partial_t$ is computed assigning weight $-1$ to $\partial_{z_j}$ and $\partial_{\overline{z}_j}$ and weight $-2$ to $\partial_t$. In other words, $E$ have parabolic weight $\geq w$ if every $a_j$ and every $b_j$ has parabolic weight $\geq w+1$ and $c$ has parabolic weight $\geq w+2$. 

We refer to \cite{folland_stein} and \cite[Theorem 3.5]{dragomir_tomassini} for a proof of the following proposition. 

\begin{lem}[Folland--Stein coordinates]\label{lem:f_s}
	Let $M$ be a $\CR$ manifold and let $x_0\in M$ be a strictly pseudoconvex point. Then in a neighborhood of $x_0$ there exists a local system of coordinates $(z,t)\in \C\times \R$ such that $(z(x_0), t(x_0))=0$ and a system of local generators of $T_{1,0}M$ of the form \begin{equation}\label{f-s}
	L_j=\partial_{z_j}+i\overline{z}_j\partial_t+ E_j\qquad (j=1,\ldots, n),
	\end{equation}
	where the error terms $E_j$ have parabolic weight $\geq 0$. 
	\end{lem}

\subsection{A density lemma}

\begin{lem}\label{density_lem}
There exists a dense subspace $\mathcal{D}$ of $\CR^2(\mathbb{H}^n,\sigma)$ such that for every $h\in \mathcal{D}$, $\alpha, \beta\in \N^n$, $\gamma\in \N$, and $N\in \N$, we have \begin{equation*}
|\partial_{z}^\alpha\partial_{\overline z}^\beta\partial_t^\gamma h(z,t)|\leq C_N(|z|^2+|t|)^{-N}.
\end{equation*}
\end{lem}

\begin{proof}
In view of the invariance under $t$-translations of the Heisenberg $\CR$ structure, we exploit the partial Fourier transform\begin{equation*}
\mathcal{F}f(z,\xi):=\frac{1}{\sqrt{2\pi}}\int_{\R}f(z,t)e^{-i\xi t}dt.
\end{equation*}
The operator $\mathcal{F}$ is a unitary isometry of $L^2(\mathbb{H}^n,\sigma)$, and \begin{equation}\label{eq:partial_fourier}
\mathcal{F}\left((\partial_{\overline{z}_j}-iz_j\partial_t)f\right)=(\partial_{\overline{z}_j}+z_j\xi)\mathcal{F}f=e^{-\xi|z|^2}\partial_{\overline{z}_j}\left(e^{\xi|z|^2}\mathcal{F}f\right).
\end{equation}
Thus $f\in \CR^2(\mathbb{H}^n,\sigma)$ if and only if $g(z,\xi)=e^{\xi|z|^2}\mathcal{F}f$ is holomorphic for every $\xi\in \R$. Since $g\in L^2(\C^n,e^{-2\xi|z|^2})$ for almost every $\xi$ and there are no nonzero Lebesgue square-integrable holomorphic functions on $\C^n$, we must have $g(z,\xi)=0$ for almost every $\xi\leq 0$. In conclusion, the operator $f\mapsto e^{\xi|z|^2}\mathcal{F}f$ establishes a unitary isomorphism between $\CR^2(\mathbb{H}^n,\sigma)$ and the Hilbert space 
\begin{equation*}\begin{aligned}
\mathcal{H}=\biggr\{g \colon & \C^n\times (0,+\infty)\rightarrow \C\colon g(\cdot, \xi) \text{ is holomorphic }\forall \xi>0 \text{ and } \\ & \int_{\C^n\times (0,+\infty)}|g|^2e^{-2\xi|z|^2}<+\infty\biggr\}.\end{aligned} 
\end{equation*}
A dense subspace $\mathcal{D}'$ of $\mathcal{H}$ is given by the linear span of functions of the form $g(z,\xi)=P(z)\varphi(\xi)$, where $P(z)=P(z_1,\ldots, z_n)$ is a holomorphic polynomial and $\varphi\in C^\infty_c(\R^+)$. To see this, notice that if $g\in \mathcal{H}$ is orthogonal to $\mathcal{D}'$, then $\eps^{-1}\int_a^{a+\eps}\int_{\C^n}g(z,\xi)\overline{P(z)}e^{-2\xi|z|^2}=0$ for every $a,\eps>0$. Letting $\eps$ tend to zero, we see that, for almost every $a>0$, $\int_{\C^n}g(z,a)P(z)e^{-2a|z|^2}=0$. By the arbitrariness of $P$ and the density of polynomials in the Fock space \cite{zhu}, we conclude that $g(z,a)=0$ for almost every $a>0$.

Let $\mathcal{D}:=\mathcal{F}^{-1}(\mathcal{D}')$. By what we just proved, $\mathcal{D}$ is dense in $\CR^2(\mathbb{H}^n)$. It is generated by elements of the form \begin{equation*}
P(z)\int_{\R}\varphi(\xi)e^{-\xi|z|^2}e^{i\xi t}d\xi
\end{equation*} with $P$ and $\varphi$ as above. If $\varphi$ is supported on $[a,+\infty)$, then a standard integration by parts shows that \begin{equation*}
\left|(|z|^2-it)^N\int \varphi(\xi)e^{-\xi|z|^2}e^{i\xi t}d\xi\right|\leq e^{-a|z|^2}\int |\varphi^{(N)}(\xi)|d\xi.
\end{equation*} 
Thus, every element of $\mathcal{D}$ decays faster than any negative power of $|z|^2+|t|$. A similar argument proves that the same holds for every derivative. 
\end{proof}

\subsection{Models satisfy property C}

The next result is not strictly needed for the proof, but shows that the noncompact model $(\mathbb{H}^n,\sigma)$ is in the class of CR manifolds to which the theorem applies. 

\begin{prop}\label{heis_propC}
	The model $(\mathbb{H}^n,\sigma)$ satisfies property C at every point. 
\end{prop}

\begin{proof} Let $\mathcal{F}$ be the partial Fourier transform as in the proof of Lemma \ref{density_lem}.  By \eqref{eq:partial_fourier}, we have the identity\begin{equation*}
		\mathcal{F}\left(\mathcal{S}_nf\right)(\cdot, \xi)=e^{-\xi|\cdot|^2}\mathcal{B}_{\xi}\left(e^{\xi|\cdot|^2}\mathcal{F}f(\cdot,\xi)\right),
	\end{equation*}
	where $\mathcal{B}_\xi$ is the Bergman projection of the Fock space $(\C^n,e^{-2\xi|z|^2})$, that is, the orthogonal projection onto the closed subspace of entire holomorphic functions that are square integrable with respect to the Gaussian measure with density $e^{-\xi|z|^2}$. In particular, $\mathcal{B}_\xi=0$ when $\xi\leq0$. The following estimate is well-known: 
	\begin{equation}\label{fock}
		\int_{\C^n} |(1-B_\xi)u|^2e^{-2\xi|z|^2}dA(z)\leq2 \xi^{-1}\int_{\C^n} \left(\sum_j|\partial_{\overline{z}_j} u|^2\right)e^{-2\xi|z|^2}dA(z)\quad\forall u\in C^\infty_c(\C^n),\ \xi>0.
	\end{equation}	
In fact, $g=(1-B_\xi)u$ is the solution of minimal Gaussian $L^2$ norm of the equation $\overline\partial g=\overline\partial u$, and one may apply \cite[Lemma 4.4.1]{hormander}. 
	
	Let $f\in C^\infty_c(\mathbb{H}^n)$. Define the projection operator $P_{\leq T}$ ($T\in \R$) by the identity\begin{equation*}
		\mathcal{F}\left(P_{\leq T}f\right)(z,\xi)=1_{(-\infty,T]}(\xi)\mathcal{F}f(z,\xi),
	\end{equation*} and put $P_{>T}:=I-P_{\leq T}$.\newline 
	Applying \eqref{fock} to $e^{\xi|\cdot|^2}\mathcal{F}f(\cdot,\xi)$ and integrating in $\xi\in(T,+\infty)$, we find\begin{equation}\label{high_freq}
		\int_{\mathbb{H}^n} |P_{> T}f-\mathcal{S}_nP_{> T}f|^2d\sigma\leq\frac{2}{T}  \int_{\mathbb{H}^n} \left(\sum_j|\overline{L}_jf|^2\right)d\sigma.
	\end{equation}
	Notice that we used the fact that $\overline{L}_j$ and $\mathcal{S}_n$ commute with $P_{>T}$. We treat the low-frequency component using the standard commutator formula\begin{equation*}
		\int_{\mathbb{H}^n}|L_jf|^2d\sigma = \int_{\mathbb{H}^n}|\overline{L}_jf|^2d\sigma+\int_{\mathbb{H}^n}[L_j,\overline{L}_j]f\cdot \overline{f}d\sigma.
	\end{equation*}
	Since $[L_j,\overline{L}_j]=-2i\partial_t$, this immediately yields\begin{equation}\label{low_freq0}
		\int_{\mathbb{H}^n}|L_jP_{\leq T}f|^2d\sigma \leq  \int_{\mathbb{H}^n}|\overline{L}_jP_{\leq T}f|^2d\sigma+2T\int_{\mathbb{H}^n}|P_{\leq T}f|^2d\sigma.
	\end{equation}
	The identity $[L_j,\overline{L}_j]=-2i\partial_t$ also implies that $\{\Re(L)_j, \Im(L)_j\}_j$ is a system of vector fields satisfying H\"ormander's bracket condition of order $2$. Hence (see, e.g., \cite{kohn}), denoting by $\chi$ and $\chi'$ two test functions with $\chi'=1$ on a neighborhood of the support of $\chi$, we have \begin{equation}\label{hormander}
		||f||_{W^{1/2}(B)}^2\leq C_B\left( ||f||_{L^2}^2+\sum_j\left(|| L_jf||_{L^2}^2+||\overline{L}_jf||_{L^2}^2\right)\right)\qquad\forall f\in C^\infty(\mathbb{H}^n),
	\end{equation} where $B\subset\C^n$ is a ball (or any, say, smooth open set) and $W^{1/2}$ denotes the fractional Sobolev norm of order $1/2$. Putting \eqref{low_freq0} and \eqref{hormander} together (and using again $[\overline{L}_j,P_{\leq T}]=0$), we obtain\begin{equation}\label{low_freq}
		||P_{\leq T}f||_{W^{1/2}}^2\leq C_B(1+2T)||f||_{L^2}^2+2C_B\sum_j||\overline{L}_jf||_{L^2}^2.
	\end{equation}
	We can now complete the proof. Given $\{f_k\}_k\subset C^\infty_c(B)$ with bounded CR energy, the sequence $\{P_{\leq T}f_k\}_k$, and a fortiori $\{(1-\mathcal{S}_n)P_{\leq T}f_k\}_k$, has a subsequence of $L^2$-diameter $\leq \eps$. This follows from \eqref{low_freq} and the compactness of the Sobolev embedding $W^\frac{1}{2}(B)\hookrightarrow L^2(B)$. If $T\geq 2\eps^{-2}$, by virtue of \eqref{high_freq} the whole sequence $\{(1-\mathcal{S}_n)P_{> T}f_k\}_k$ has $L^2$-diameter $\leq\eps$ too. The conclusion of the proposition follows immediately. 
\end{proof}

\section{Proof of Theorem \ref{thm:intro}}\label{sec:proof}

Write $\Sz=\Sz_{M,\nu}$ for simplicity. Fix a coordinate patch $\Omega$ around $x_0$ equipped with a system of Folland--Stein coordinates $(z,t)$, and $p\in [1,+\infty)$ (notice that $\Sz$ is self-adjoint, so the case $p=\infty$ of the theorem follows by duality from the $p=1$ case).  

We identify $\Omega$ with an open neighborhood of $(0,0)$ in $\mathbb{H}^n\equiv \C^n\times \R$ via these coordinates.

Fix $f\in C^\infty_c(\mathbb{H}^n)$. If $\lambda>0$ is large enough, depending on the support of $f$, $f\circ \Phi_\lambda$ is supported on $\Omega$, and hence it may be thought of as a test function on $M$. Without loss of generality, we may assume that $\frac{d\nu}{d\sigma}(x_0)=1$, where $\sigma$ is Lebesgue measure in the Folland--Stein coordinates.

Thus $\Sz(f\circ \Phi_\lambda)$ is a well-defined element of $L^2(M,\nu)$. Let $\rho\in C^\infty_c(\mathbb{H}^n,[0,1])$ be a cut-off function identically equal to $1$ in a neighborhood of the origin, i.e., of $x_0$. We set $\rho_\mu:=\rho\circ \Phi_{\mu}$ ($\mu>0$). Notice that the support of $\rho_\mu$ shrinks to $x_0$ as $\mu$ tends to $\infty$, and that $\Sz(f\circ \Phi_\lambda)\rho_\mu$ is supported on $\Omega$ for every large $\mu$. Thus, it can be thought of as a function on $\mathbb{H}^n$. Put \begin{equation*}
	g_\lambda:=\left(\Sz(f\circ \Phi_\lambda)\rho_{\sqrt{\lambda}}\right)\circ \Phi_{\lambda^{-1}}\in L^2(\mathbb{H}^n,\sigma).
\end{equation*}

\begin{rmk} It will be clear soon that $\sqrt{\lambda}$ could be replaced by any function $F:\R^+\rightarrow \R^+$ such that
\begin{equation*}\lim_{\lambda\rightarrow+\infty}F(\lambda)=+\infty\quad \text{and}\quad  \lim_{\lambda\rightarrow+\infty}\frac{F(\lambda)}{\lambda}=0.
\end{equation*}
In other words, what turns out to be crucial is localizing the Szegő projection of $f\circ \Phi_\lambda$, which is supported at the parabolic infinitesimal scale $\lambda^{-1}$, at an infinitesimal scale much larger than $\lambda^{-1}$. 
\end{rmk}

It is clear that for every $\lambda$ large\begin{eqnarray*}
	\int_{\mathbb{H}^n}|g_\lambda|^pd\sigma &=& 	\lambda^{2n+2}\int_{\Omega}|\Sz(f\circ \Phi_\lambda)\rho_{\sqrt\lambda}|^pd\sigma\\
	&=& (1+o(1))\lambda^{2n+2}\int_{\Omega}|\Sz(f\circ \Phi_\lambda)\rho_{\sqrt\lambda}|^pd\nu\\
	&\leq& (1+o(1))\lambda^{2n+2}N_p(M,\nu)^p\int_{\Omega}|f\circ \Phi_\lambda|^pd\nu\\
	&=& (1+o(1))\lambda^{2n+2}N_p(M,\nu)^p\int_{\Omega}|f\circ \Phi_\lambda|^pd\sigma\\
	&=& (1+o(1))N_p(M,\nu)^p\int_{\mathbb{H}^n}|f|^pd\sigma,
\end{eqnarray*}
where we used a couple of times the fact that $\frac{d\nu}{d\sigma}(x_0)=1$. 

By Banach--Alaoglu, along a diverging subsequence of $\lambda$'s, $g_\lambda$ has a weak limit both in $L^2(\mathbb{H}^n,\sigma)$ and $L^p(\mathbb{H}^n,\sigma)$, for any fixed $p>1$. Denoting by $g\in L^2\cap L^p(\mathbb{H}^n,\sigma)$ this limit, we clearly have \begin{equation}\label{norm_bound}\int_{\mathbb{H}^n}|g|^pd\sigma\leq N_p(M,\nu)^p\int_{\C^n\times \R}|f|^pd\sigma.\end{equation} 
From now on, limits in $\lambda$ are always along appropriate diverging subsequences. 

\par \emph{We claim that the limit $g$ is $\CR$ with respect to the model structure}, i.e., that $g\in \CR^2(\mathbb{H}^n,\sigma)$. To see this, we start by noticing that by \eqref{f-s},\begin{equation}\label{f-s-2}
\overline{L}_j\left(\Sz(f\circ \Phi_\lambda)\rho_{\sqrt\lambda}\right)\circ \Phi_{\lambda^{-1}}=\lambda (\partial_{z_j}+i\overline{z}_j\partial_t)g_\lambda+\lambda E_j^\lambda g_\lambda,
\end{equation}
in the sense of distributions. If $E_j=\sum_{k=1}^n\left\{a_{j,k}\partial_{z_k}+b_{j,k}\partial_{\overline{z}_k}\right\}+c\partial_t$, the rescaled error terms above are given by \begin{equation*}
E_j^\lambda=\sum_{k=1}^n\left\{a_{j,k}\circ\Phi^{-1}_\lambda\cdot \partial_{z_k}+b_{j,k}\circ\Phi^{-1}_\lambda\cdot \partial_{\overline{z}_k}\right\}+\lambda c\circ\Phi^{-1}_\lambda\cdot\partial_t.
\end{equation*}
Notice that the formal adjoint of $E_j^\lambda$ (w.r.t.~Lebesgue measure $d\sigma$) is $-E_j^\lambda-\lambda^{-1}e_j\circ\Phi_\lambda^{-1}$, where $e^\lambda_j(z,t)$ is a smooth function. 

Since $\rho$ is identically $1$ on a neighborhood $V$ of the origin, the LHS of \eqref{f-s-2} vanishes on $\Phi_{\sqrt{\lambda}}(V)$. Since $\bigcup_{\lambda}\Phi_{\sqrt{\lambda}}(V)=\C^n\times \R$, for any $\varphi\in C^\infty_c(\mathbb{H}^n)$ and $\lambda$ large, we have
\begin{eqnarray*}
\int_{\C^n\times \R} (\partial_{z_j}+i\overline{z}_j\partial_t)g_\lambda\cdot \varphi\, d\sigma&=&-\int_{\C^n\times \R} E_j^\lambda g_\lambda\cdot \varphi\, d\sigma\\
&=&\int_{\C^n\times \R}g_\lambda\cdot  (E_j^\lambda +\lambda^{-1}e_j\circ\Phi_\lambda^{-1})\varphi\, d\sigma.
\end{eqnarray*}
By Lemma \ref{lem:f_s}, $a_{j,k}$ and $b_{j,k}$ have parabolic weight $\geq 1$ and $c$ has weight $\geq 2$, and therefore $||(E_j^\lambda +\lambda^{-1}e_j\circ\Phi_\lambda^{-1})\varphi||_{\infty}=O(\lambda^{-1})$. Since $g_\lambda$ are uniformly in $L^2(\C^n\times \R)$, we conclude that $(\partial_{z_j}+i\overline{z}_j\partial_t)g_\lambda$ tends to zero in the sense of distributions. Thus, $(\partial_{z_j}+i\overline{z}_j\partial_t)g=0$ and the claim is proved. 
\newline
\par We are left with the proof that $g=\Sz_n(f)$. In fact, if we show this then \eqref{norm_bound} and the arbitrariness of $f$ entail $N_p(\mathbb{H}^n,\sigma)\leq N_p(M,\nu)$, as we wanted.

Since we already know that $g$ is $\CR$, to prove that $g=\Sz_n(f)$ what we need to show is that $f-g$ is orthogonal to $\CR^2(\mathbb{H}^n,\sigma)$. This is clearly equivalent to \begin{equation*}
\lim_{\lambda\rightarrow +\infty}\int_{\mathbb{H}^n}g_\lambda \overline{h}d\sigma = \int_{\mathbb{H}^n}f \overline{h}d\sigma \qquad\forall h\in \CR^2(\mathbb{H}^n,\sigma). 
\end{equation*}
In fact, it is enough to verify this identity for $h$ in the dense subspace $\mathcal{D}$ of Lemma \ref{density_lem}. By the self-adjunction of $\Sz$ in $L^2(M,\nu)$, we have \begin{eqnarray*}
\int_{\mathbb{H}^n}g_\lambda \overline{h}d\sigma&=&\int_{\mathbb{H}^n}\left(\Sz(f\circ \Phi_\lambda)\rho_{\sqrt\lambda}\right)\circ \Phi_{\lambda^{-1}}\cdot \overline{h}d\sigma\\&=&
\lambda^{2n+2}\int_M \Sz(f\circ \Phi_\lambda)\rho_{\sqrt\lambda} \cdot \overline{h\circ \Phi_{\lambda}}\frac{d\sigma}{d\nu}d\nu\\&=&
\lambda^{2n+2}\int_\Omega f\circ \Phi_\lambda \cdot \overline{\Sz\left(\rho_{\sqrt\lambda}\frac{d\sigma}{d\nu}\cdot h\circ \Phi_{\lambda}\right)}\frac{d\nu}{d\sigma}d\sigma\\&=&
\int_{\mathbb{H}^n} f \cdot \overline{\left\{\Sz\left(\rho_{\sqrt\lambda}\frac{d\sigma}{d\nu}\cdot h\circ \Phi_{\lambda}\right)\frac{d\nu}{d\sigma}\right\}\circ \Phi_\lambda^{-1}} d\sigma
\end{eqnarray*}
Notice that the various passages from $d\sigma$ to $d\nu$ are meaningful, because for $\lambda$ large, $\rho_\lambda$ and $f\circ \Phi_\lambda$ are supported in $\Omega$.

Thus, our task is reduced to proving that  \begin{equation*}
\lim_{\lambda\rightarrow+\infty} \left\{\Sz\left(\rho_{\sqrt\lambda}\frac{d\sigma}{d\nu}\cdot h\circ \Phi_{\lambda}\right)\frac{d\nu}{d\sigma}\right\}\circ \Phi_\lambda^{-1} = h\qquad\forall h\in \mathcal{D}
\end{equation*}
in the sense of distributions.\newline 
Notice that if we remove $\Sz$ from the expression in the limit, we get $h\rho_{\lambda^{-\frac{1}{2}}}$, which clearly $L^2$-converges to $h$. Thus, setting \begin{equation*}
u_\lambda:=\left(I-\Sz\right)\left(\rho_{\sqrt\lambda}\frac{d\sigma}{d\nu}\cdot h\circ \Phi_{\lambda}\right),
\end{equation*} 
it is enough to see that $\lim_{\lambda\rightarrow+\infty}\left(u_\lambda \frac{d\nu}{d\sigma}\right)\circ \Phi_\lambda^{-1}=0$ in the sense of distributions, i.e., 
\begin{equation*}
\int_{\mathbb{H}^n}u_\lambda\frac{d\nu}{d\sigma}\lambda^{2n+2}\varphi\circ \Phi_\lambda d\sigma=0\qquad\forall \varphi\in C^\infty_c(\mathbb{H}^n). 
\end{equation*} 
Let $B$ be the compact neighborhood of $x_0$ appearing in the definition of Property C. We claim that:\begin{equation}\label{final_claim}
1_B\lambda^{n+1}u_\lambda\text{ has a strong } L^2 \text{ limit }v,\text{ along an appropriate subsequence}.\end{equation} 
Notice that $\nu$ and $\sigma$ are comparable on $B$ and we do not need to specify in which $L^2$ norm the convergence happens. \newline
Let us show that the claim allows to conclude. We write
\begin{eqnarray*}
&&\int_{\mathbb{H}^n}u_\lambda\frac{d\nu}{d\sigma}\lambda^{2n+2}\varphi\circ \Phi_\lambda\, d\sigma=\int_{\mathbb{H}^n}(\lambda^{n+1}u_\lambda-v)\frac{d\nu}{d\sigma}\lambda^{n+1}\varphi\circ \Phi_\lambda\, d\sigma\\
&&+\int_{\mathbb{H}^n}(v-v_0)\frac{d\nu}{d\sigma}\lambda^{n+1}\varphi\circ \Phi_\lambda\, d\sigma+\lambda^{-n-1}\int_{\mathbb{H}^n}v_0\frac{d\nu}{d\sigma}\lambda^{2n+2}\varphi\circ \Phi_\lambda\, d\sigma,
\end{eqnarray*}
where $v_0\in C_c(\mathbb{H}^n)$ is such that $||v-v_0||_{L^2}\leq \eps$, with $\eps$ small. The first term vanishes in the limit, simply because $\lambda^{n+1}\varphi\circ \Phi_\lambda$ is uniformly bounded in $L^2$ norm, and the second is $\mathcal{O}(\eps)$ for the same reason. Since the last one is asymptotic to $\lambda^{-n-1}v_0(0,0)\int\varphi d\sigma$, the conclusion follows by the arbitrariness of $\eps$. 
 
 We now prove claim \eqref{final_claim}. In virtue of property C, it is enough to prove that the family of test functions $\{\lambda^{n+1}\rho_{\sqrt{\lambda}}\frac{d\sigma}{d\nu}\cdot h\circ \Phi_{\lambda}\}_\lambda$ is bounded w.r.t.~the "CR energy" \eqref{eq:CR_energy}. The $L^2$ norm is clearly bounded uniformly in $\lambda$. We compute, for $L_j$ as in Lemma \ref{lem:f_s}, 
\[\begin{aligned}
 	\overline{L}_j\left(\rho_{\sqrt\lambda}\frac{d\sigma}{d\nu}\cdot h\circ \Phi_{\lambda}\right) &=\left[\overline{L}_j\left(\rho\circ\Phi_{\sqrt\lambda}\right)\frac{d\sigma}{d\nu}+\rho_{\sqrt\lambda}\overline{L}_j\left(\frac{d\sigma}{d\nu}\right)\right] h\circ \Phi_{\lambda}+\rho_{\sqrt\lambda}\frac{d\sigma}{d\nu}\overline{L}_j\left(h\circ \Phi_{\lambda}\right)\\
 	&=\left[\left({\sqrt\lambda} \left((\partial_{z_j}+i\overline{z}_j\partial_t)\rho+E_j^{\sqrt\lambda} \rho\right)\circ \Phi_{\sqrt\lambda} \right)\frac{d\sigma}{d\nu}+\rho_{\sqrt\lambda}\overline{L}_j\left(\frac{d\sigma}{d\nu}\right)\right] h\circ \Phi_{\lambda}\\
 	& \qquad +\rho_{\sqrt\lambda}\frac{d\sigma}{d\nu}\lambda\left\{(\partial_{z_j}+i\overline{z}_j\partial_t)h+E_j^\lambda h\right\}\circ \Phi_{\lambda}\\
 	&=\left[\left( {\sqrt\lambda} \left((\partial_{z_j}+i\overline{z}_j\partial_t)\rho+E_j^{\sqrt\lambda} \rho\right)\circ \Phi_{\sqrt\lambda} \right)\frac{d\sigma}{d\nu}+\rho_{\sqrt\lambda}\overline{L}_j\left(\frac{d\sigma}{d\nu}\right)\right] h\circ \Phi_{\lambda}\\
 	&\qquad+\rho_{\sqrt\lambda}\frac{d\sigma}{d\nu}\lambda 	\left(E_j^\lambda h\right)\circ \Phi_{\lambda} ,	\end{aligned}
 \]
where we used the fact that $h$ is $\CR$ with respect to the Heisenberg structure. The easiest term to deal with is \begin{equation*}
\int_M\left|h\circ \Phi_{\lambda}\rho_{\sqrt\lambda}\overline{L}_j\left(\frac{d\sigma}{d\nu}\right)\right|^2d\nu= (1+o(1))\left|\overline{L}\left(\frac{d\sigma}{d\nu}\right)(0)\right|^2\lambda^{-2n-2}\int_{\mathbb{H}^n}|h|^2d\sigma.
\end{equation*}
Next, 
\begin{equation*}\begin{aligned}
\int_M&\left|\sqrt\lambda \left( (\partial_{z_j}+i\overline{z}_j\partial_t)\rho+E_j^{\sqrt\lambda} \rho\right)\circ \Phi_{\sqrt\lambda}\frac{d\sigma}{d\nu}\right|^2|h\circ\Phi_\lambda|^2d\nu\\
&=(1+o(1))\lambda^{-n-1}\lambda\int_{\mathbb{H}^n}\left|\left( (\partial_{z_j}+i\overline{z}_j\partial_t)\rho+E_j^{\sqrt\lambda} \rho\right)\right|^2|h\circ\Phi_{\sqrt\lambda}|^2d\sigma\\
&\leq C(1+o(1))\lambda^{-n}\int_{\mathbb{H}^n\setminus V}|h\circ\Phi_{\sqrt\lambda}|^2d\sigma,\end{aligned}
\end{equation*}
where we used the fact that $\rho$ is identically equal to $1$ in the neighborhood $V$ of the origin and that the coefficients of $E_j^{\sqrt{\lambda}}$ are bounded uniformly in $\lambda$. Here we take advantage of the fact that $h$ is in $\mathcal{D}$. By Lemma \ref{density_lem}, the quantity above may be estimated by $|h\circ\Phi_{\sqrt\lambda}|\leq C_N\lambda^{-N}(|z|^2+|t|)^{-N}$ for any $N\in \N$. Choosing $N$ large enough, we obtain the desired estimate. \newline
Finally, \begin{equation*}
\int_M\left|\rho_{\sqrt\lambda}\frac{d\sigma}{d\nu}\lambda 	\left(E_j^\lambda h\right)\circ \Phi_{\lambda}\right|^2d\nu=(1+o(1))\lambda^{-2n-2}\int_{\mathbb{H}^n}|\lambda \left(E_j^\lambda h\right)\rho\circ \Phi_{\lambda^{-\frac{1}{2}}}|^2\, d\sigma.
\end{equation*}
The vector field $\lambda E_j^\lambda$ has coefficients bounded uniformly in $\lambda$. Therefore, to bound the last term we need $\int_{\mathbb{H}^n}\sum_{j=1}^n(|\partial_{z_j}h|^2+|\partial_{\overline z_j}h|^2)+|\partial_th|^2<+\infty$, which certainly holds for $h\in \mathcal{D}$. This completes the proof of the boundedness of the CR energy of $\{\lambda^{n+1}\rho_{\sqrt{\lambda}}\frac{d\sigma}{d\nu}\cdot h\circ\Phi_\lambda \}_\lambda$, and hence of the theorem. 

\bibliographystyle{amsalpha}
\bibliography{extremality}

\providecommand{\bysame}{\leavevmode\hbox to3em{\hrulefill}\thinspace}
\providecommand{\MR}{\relax\ifhmode\unskip\space\fi MR }
\providecommand{\MRhref}[2]{%
  \href{http://www.ams.org/mathscinet-getitem?mr=#1}{#2}
}
\providecommand{\href}[2]{#2}
\begin{thebibliography}{NRSW89}

\bibitem[Bar88]{barrett}
David~E. Barrett, \emph{A remark on the global embedding problem for
  three-dimensional {CR} manifolds}, Proceedings of the American Mathematical
  Society \textbf{102} (1988), no.~4, 888--892.

\bibitem[Bar12a]{baracco2}
Luca Baracco, \emph{Erratum to: {T}he range of the tangential
  {C}auchy-{R}iemann system to a {CR} embedded manifold [mr2981820]}, Invent.
  Math. \textbf{190} (2012), no.~2, 511--512. \MR{2981821}

\bibitem[Bar12b]{baracco}
\bysame, \emph{The range of the tangential {C}auchy-{R}iemann system to a {CR}
  embedded manifold}, Invent. Math. \textbf{190} (2012), no.~2, 505--510.
  \MR{2981820}

\bibitem[BER99]{ber}
M.~Salah Baouendi, Peter Ebenfelt, and Linda~Preiss Rothschild, \emph{Real
  {S}ubmanifolds in {C}omplex {S}pace and {T}heir {M}appings}, Princeton
  Mathematical Series, vol.~47, Princeton University Press, Princeton, NJ,
  1999. \MR{1668103}

\bibitem[CD06]{charpentier_dupain}
Philippe Charpentier and Yves Dupain, \emph{Estimates for the {B}ergman and
  {S}zeg\"{o} projections for pseudoconvex domains of finite type with locally
  diagonalizable {L}evi form}, Publ. Mat. \textbf{50} (2006), no.~2, 413--446.
  \MR{2273668}

\bibitem[Dal22]{dallara_L1}
Gian~Maria Dall'Ara, \emph{Around {$L^1$} (un)boundedness of {B}ergman and
  {S}zeg\"{o} projections}, J. Funct. Anal. \textbf{283} (2022), no.~5, Paper
  No. 109550, 27. \MR{4429576}

\bibitem[DT07]{dragomir_tomassini}
Sorin Dragomir and Giuseppe Tomassini, \emph{Differential {G}eometry and
  {A}nalysis on {CR} {M}anifolds}, vol. 246, Springer Science \& Business
  Media, 2007.

\bibitem[FS74]{folland_stein}
G.~B. Folland and E.~M. Stein, \emph{Estimates for the {$\bar \partial _{b}$}
  complex and analysis on the {H}eisenberg group}, Comm. Pure Appl. Math.
  \textbf{27} (1974), 429--522. \MR{367477}

\bibitem[H{\"o}r90]{hormander}
Lars H{\"o}rmander, \emph{An {I}ntroduction to {C}omplex {A}nalysis in
  {S}everal {V}ariables}, third ed., North-Holland Mathematical Library,
  vol.~7, North-Holland Publishing Co., Amsterdam, 1990. \MR{1045639}

\bibitem[Koh73]{kohn}
J.~J. Kohn, \emph{Pseudo-differential operators and hypoellipticity}, Partial
  differential equations ({P}roc. {S}ympos. {P}ure {M}ath., {V}ol. {XXIII},
  {U}niv. {C}alifornia, {B}erkeley, {C}alif., 1971), 1973, pp.~61--69.
  \MR{0338592}

\bibitem[Koh85]{kohn_estimates}
\bysame, \emph{Estimates for {$\bar\partial_b$} on pseudoconvex {CR}
  manifolds}, Pseudodifferential operators and applications ({N}otre {D}ame,
  {I}nd., 1984), Proc. Sympos. Pure Math., vol.~43, Amer. Math. Soc.,
  Providence, RI, 1985, pp.~207--217. \MR{812292}

\bibitem[Liu18]{liu}
Congwen Liu, \emph{Norm estimates for the {B}ergman and {C}auchy-{S}zeg\"{o}
  projections over the {S}iegel upper half-space}, Constr. Approx. \textbf{48}
  (2018), no.~3, 385--413. \MR{3869446}

\bibitem[MS97]{mcneal_stein}
J.~D. McNeal and E.~M. Stein, \emph{The {S}zeg{\"o} projection on convex
  domains}, Math. Z. \textbf{224} (1997), no.~4, 519--553. \MR{1452048}

\bibitem[NRSW89]{nagel_rosay_stein_wainger}
Alexander Nagel, Jean-Pierre Rosay, Elias~M. Stein, and Stephen Wainger,
  \emph{Estimates for the {B}ergman and {S}zeg{\"o} kernels in $\mathbb{C}^2$},
  Annals of Mathematics (1989), 113--149.

\bibitem[PS77]{phong_stein}
Duong~H. Phong and Elias~M. Stein, \emph{Estimates for the {B}ergman and
  {S}zeg{\"o} projections on strongly pseudo-convex domains}, Duke Mathematical
  Journal \textbf{44} (1977), no.~3, 695--704.

\bibitem[Ste93]{stein}
Elias~M. Stein, \emph{Harmonic {A}nalysis: {R}eal-variable {M}ethods,
  {O}rthogonality, and {O}scillatory {I}ntegrals}, Princeton Mathematical
  Series, vol.~43, Princeton University Press, Princeton, NJ, 1993, With the
  assistance of Timothy S. Murphy, Monographs in Harmonic Analysis, III.
  \MR{1232192}

\bibitem[Str10]{straube}
Emil~J. Straube, \emph{{L}ectures on the {$L^2$}-{S}obolev {T}heory of the
  {$\bar\partial$}-{N}eumann {P}roblem}, ESI Lectures in Mathematics and
  Physics, European Mathematical Society (EMS), Z\"{u}rich, 2010.

\bibitem[Zhu12]{zhu}
Kehe Zhu, \emph{Analysis on {F}ock {S}paces}, Graduate Texts in Mathematics,
  vol. 263, Springer, New York, 2012. \MR{2934601}

\end{thebibliography}

\end{document}